\documentclass[reqno, 12pt, oneside]{amsart}
\usepackage{amsmath, amsthm, amssymb, enumerate, color}
\usepackage[multiple]{footmisc}

\theoremstyle{plain}
\newtheorem{THM}{Theorem}
\newtheorem*{Penrose}{The Riemannian Penrose Inequality Conjecture}
\newtheorem{lemma}{Lemma}[section]
\newtheorem{prop}[lemma]{Proposition}
\newtheorem{thm}[lemma]{Theorem}

\newtheorem{defi}[lemma]{Definition}
\newtheorem{cor}[lemma]{Corollary}

\theoremstyle{definition}
\newtheorem*{ack}{Acknowledgements}

\newtheorem{remark}[lemma]{Remark}
\newtheorem*{notation}{Notation}
\newtheorem*{pfPenrose}{Proof of Theorem~\ref{th:Penrose}}
\newtheorem*{pfthm}{Proof of Theorem~\ref{th:equality}}
\newtheorem*{pf-mean-convexity}{Proof of Theorem~\ref{th:mean-convexity}}
\newtheorem*{pf-mean-convexity-special}{Proof of Theorem~\ref{th:mean-convexity} (special case)}
\newtheorem*{pfSchwarzschild}{Proof of Theorem~\ref{th:Schwarzschild}}

\numberwithin{equation}{section}

\theoremstyle{remark}

\newcommand{\p}{\partial}

\newenvironment{enumeratei}{\begin{enumerate}[\upshape (i)]}{\end{enumerate}}

\begin{document}
\title[The equality case of the Penrose Inequality]{The equality case of the Penrose inequality for asymptotically flat graphs}
\author{Lan-Hsuan Huang}
\address{Department of Mathematics\\
University of Connecticut\\
Storrs, CT 06269, USA}
\email{lan-hsuan.huang@uconn.edu}

\author{Damin Wu}
\address{Department of Mathematics\\
University of Connecticut\\
Storrs, CT 06269, USA}
\email{damin.wu@uconn.edu}

\thanks{The first author acknowledges NSF grant DMS-$1005560$ and DMS-$1301645$ for partial support.}

\date{\today}

\begin{abstract}
We prove the equality case of the Penrose inequality in all dimensions for asymptotically flat  hypersurfaces. It was recently proven by G.~Lam~\cite{Lam} that the Penrose inequality holds for asymptotically flat graphical hypersurfaces in Euclidean space with non-negative scalar curvature and with a minimal boundary. Our main theorem states that if the equality holds, then the hypersurface is a Schwarzschild solution. As part of our proof,  we show that asymptotically flat graphical hypersurfaces with a minimal boundary and non-negative scalar curvature must be mean convex, using the argument that we developed in~\cite{Huang-Wu}. This enables us to obtain the ellipticity for the linearized scalar curvature operator and to establish the strong maximum principles for the scalar curvature equation. 
\end{abstract}

\maketitle
\section{Introduction}
The Penrose inequality in general relativity states that the ADM mass of an asymptotically flat manifold is at least the mass of the black holes that it contains, if the energy density is non-negative everywhere. A particularly important special case of the physical statement is called the Riemannian Penrose inequality.

\begin{Penrose}
Let $(M^n, g), n\ge 3$, be an asymptotically flat $n$-dimensional smooth manifold with a strictly outer-minimizing smooth minimal boundary which is compact (not necessarily connected) of total $(n-1)$-volume $A$. Suppose  $M$ has non-negative scalar curvature and ADM mass $m$. Then
\[
	m \ge \frac{1}{2} \left(\frac{A}{\omega_{n-1}} \right)^{\frac{n-2}{n-1}},
\]
where $\omega_{n-1}$ is the volume of the unit $(n-1)$-sphere in Euclidean space. Moreover, the equality holds if and only if $(M,g)$ is isometric to the region of a Schwarzschild metric outside its minimal hypersurface.
\end{Penrose}

G.~Huisken and T.~Ilmanen proved the conjecture in dimension three for a connected minimal boundary~\cite{Huisken-Ilmanen}. H. Bray used a different approach and proved the conjecture in dimension three for any number of components of the minimal boundary~\cite{Bray-Penrose}. In dimensions less than $8$, the inequality was proved by H.~Bray and D.~Lee, with the extra spin assumption for the equality case~\cite{Bray-Lee}. In the case that $(M,g)$ is conformally flat, H.~Bray and K.~Iga derived new properties of superharmonic functions in $\mathbb{R}^n$ and proved the Penrose inequality with a suboptimal constant for $n=3$~\cite{Bray-Iga},  F.~Schwartz obtained a lower bound of the ADM mass in terms of the Euclidean volume of the region enclosed by  the minimal boundary~\cite{Schwartz}, and J.~Jauregui proved a Penrose-like inequality~\cite{Jauregui}. For the Penrose inequality (with the sharp constant) in dimensions higher than $8$, the only result that we know, other than the spherically symmetric case, is the result of G.~Lam~\cite{Lam} (cf. \cite{deLima-Girao:2012}), where he proved that the Penrose \emph{inequality} for graphical asymptotically flat hypersurfaces. (Note some related work regarding the Penrose inequality for asymptotically hyperbolic graphs in~\cite{Dahl-Gicquaud-Sakovich, deLima-Girao0}.)

\begin{THM}[\cite{Lam}] \label{th:Penrose}
Let $\Omega \subset \mathbb{R}^n, n\ge 3,$ be open and bounded. Assume that either each connected component of $\Omega$ is star-shaped or $\partial \Omega$ is outer-minimizing\footnote{In \cite{Lam}, each connected component of $\Omega$ was assumed convex, but the proof can be generalized to our setting.}\footnote{The boundary $\partial \Omega$ is called outer-minimizing if whenever $\Omega'$ is a domain with $\Omega \subset \Omega'$, then $|\partial \Omega| \le | \partial \Omega'|$.}.

Let $f\in C^2 (\mathbb{R}^n \setminus \overline{\Omega}) \cap C^0 (\mathbb{R}^n \setminus \Omega)$ be asymptotically flat. We assume that the graph of $f$ has non-negative scalar curvature.  Suppose  each connected component  of $\p \Omega$ is the level set of $f$ with $|Df(x)| \to \infty$ as $x\to \partial \Omega$, and each component of $\p \Omega$ has positive (Euclidean) mean curvature in the hyperplane.  Then,
\[
	m \ge \frac{1}{2} \left(\frac{| \p \Omega|}{\omega_{n-1}} \right)^{\frac{n-2}{n-1}},
\]
where $| \p \Omega|$ is the $(n-1)$-total volume of $\p \Omega$.
\end{THM}
The proof is simple and elegant, which we include in Section~\ref{se:main}. However, the equality case was not discussed in~\cite{Lam}, and the techniques there seem far from sufficient to handle the equality case. Our main result in this article proves the equality case in all dimensions $n\ge 3$. It may be particularly interesting because there was no rigidity result  for the Penrose inequality, other than the spherically symmetric case, known to hold in dimensions $n \ge 8$.

\begin{THM} \label{th:equality}
Under the conditions of Theorem~\ref{th:Penrose}, suppose the graph of $f$ is $C^{n+1}$ up to boundary, and 
\begin{align} \label{eq:infinity}
\max_{|x|=r} f(x) \le  \min_{|x|=r} f(x) + C \quad \mbox{ for  $n =3$ or $4$} 
\end{align} for all $r$ sufficiently large.
If the equality holds, i.e., 
\[
	m =  \frac{1}{2} \left(\frac{|\p \Omega|}{\omega_{n-1}} \right)^{\frac{n-2}{n-1}},
\]
then the graph of $f$ is identical to the region of the Schwarzschild solution of mass $m$ outside its minimal $(n-1)$-hypersurface. 
\end{THM}

\begin{remark}
In dimensions less than $8$, the above theorem is implied by more general results in \cite{Huisken-Ilmanen, Bray-Penrose, Bray-Lee} because hypersurfaces in Euclidean space are spin. Our proof is different and works for all dimensions. The additional assumption \eqref{eq:infinity} for $n=3$ or $4$ ensures that the oscillation of $f$ at infinity is under control. We actually conjecture a stronger statement that an $n$-dimensional asymptotically flat hypersurface with zero scalar curvature has the following expansion at infinity
\begin{align}
	\begin{split} \label{eq:asymptotics}
	f(x)= \left\{ \begin{array}{ll} C_0 \sqrt{|x|} + C_1+ o(1) & \mbox{ if } n =3\\
	C_0 \ln |x| + C_1+ o(1) & \mbox{ if } n =4, \end{array} \right.
	\end{split}
\end{align}
for some constants $C_0, C_1$. This conjecture should compare with the celebrated work of R.~Schoen on the uniqueness of catenoids~\cite{Schoen}, in which a preliminary result says that complete minimal hypersurfaces have specific asymptotics at infinity, up to lower order terms.  In general, hypersurfaces with zero scalar curvature are  more difficult to analyze than minimal hypersurfaces, because the scalar curvature equation of the graphing function is fully nonlinear. Assuming  the strict ellipticity and certain asymptotic behavior of the hypersurfaces at infinity (in all dimensions, which is in particular stronger than \eqref{eq:asymptotics} in the low dimensions), J.~Hounie and M.~Leite proved the uniqueness of embedded  scalar-flat hypersurfaces with two ends~\cite{Hounie-Leite}. 
\end{remark}

Our proof of Theorem~\ref{th:equality} relies on a key observation that an asymptotically flat graphical hypersurface with a minimal boundary and with non-negative scalar curvature must be mean convex. It is inspired by our earlier work \cite{Huang-Wu}, in which we proved that \emph{closed} or  \emph{complete} asymptotically flat hypersurfaces with non-negative scalar curvature must be mean convex. 
 
\begin{THM} \label{th:mean-convexity}
Let $\Omega$ be an open and bounded subset (not necessarily connected) in $\mathbb{R}^n$.  Let $f\in C^{n+1}(\mathbb{R}^n \setminus \overline{\Omega}) \cap C^0 (\mathbb{R}^n \setminus \Omega)$ be asymptotically flat and let the graph of $f$ be $C^{n+1}$ up to boundary. Suppose each connected component of $\p \Omega$ is the level set of $f$ with $|Df(x)| \to \infty$ as $x\to \partial \Omega$.  If the scalar curvature of the graph of $f$ is non-negative, then its mean curvature $H$ has a sign, i.e., either $H\ge 0$ or  $H\le 0$ everywhere. 
\end{THM}

The mean convexity enables us to derive the maximum principles for the scalar curvature equation and to compare the graph of $f$ with the Schwarzschild graph. The proof of Theorem~\ref{th:equality} is more delicate in the case $n=3$ or $4$, because the graphing function of the Schwarzschild solution tends to infinity as $|x| \to \infty$, and it is subtle to compare two unbounded graphs. To control the asymptotical behavior of $f$, we use its asymptotic flatness and develop a global strong maximum principle (Theorem~\ref{th:global-maximum}) in the region where $|x|$ is sufficiently large. The maximum principles for the scalar curvature equation are established in Section~\ref{se:maximum}.

Note that in our earlier work~\cite{Huang-Wu}, we proved the Positive Mass Theorem for hypersurfaces in Euclidean space in all dimensions, including the rigidity statement, which is a direct consequence of our proof of the positive mass inequality.  However, the proof for the Penrose case requires a new argument which uses the the \emph{strict} ellipticity of the Schwarzschild solutions of $m>0$.

In response to an interesting question raised by Christina Sormani and Dan Lee about the hypersurface which is a Schwarzschild solution outside a compact set, we have the following result.
\begin{THM} \label{th:Schwarzschild}
There is no complete $C^{n+1}$ hypersurface of one end with zero scalar curvature in $\mathbb{R}^{n+1}$ which is identical to a Schwarzschild solution with $m>0$ outside a compact set.\footnote{Note that a Schwarzschild solution may not be uniquely embedded in Euclidean space as a hypersurface. Here, we say a hypersurface is identical to a Schwarzschild solution outside a compact set  in the sense that the hypersurface is the graph of $h$ outside a compact set of a hyperplane, where $h$ is the radially symmetric function that gives a Schwarzschild solution in Proposition~\ref{pr:Schwarzschild}.}

In other words, hyperplanes are the only  complete one-ended scalar-flat $C^{n+1}$ hypersurfaces  in Euclidean space which are rotationally symmetric outside a compact set. 
\end{THM}
The above theorem is in contrast to a general result of J. Corvino~\cite{Corvino}, where he constructed the complete asymptotically flat manifold with zero scalar curvature which is a Schwarzschild metric outside a compact set, but not identical to a Schwarzschild solution. 

After this article was written and has been distributed among a small mathematics community, we noticed a preprint by L.~de Lima and F.~Gir\~{a}o \cite{deLima-Girao} that  announces the rigidity theorem, using the uniqueness result of Hounie-Leite \cite{Hounie-Leite} and assuming ellipticity and certain expansion of the graph at infinity. Our theorem is different -- the important part of our proof is to \emph{derive} ellipticity and to apply the maximum principles. We believe that our arguments will have more future applications to hypersurfaces in space forms with the appropriate scalar curvature condition (see, for example, \cite{Dahl-Gicquaud-Sakovich, deLima-Girao:2012b}). 

\begin{ack}
We thank Hugh Bray and Romain Gicquaud for bringing the question about the equality case of the Penrose inequality to our attention and the referee for pointing out the relevant references. We also thank Dan Lee and Professor S.~T.~Yau for helpful discussions. 
\end{ack}

\section{Definitions, notation, and preliminary results} 

\begin{defi} \label{de:AF}
Let $\Omega$ be a bounded subset in $\mathbb{R}^n$, $n\ge 3$. We say that $f\in C^2 (\mathbb{R}^n \setminus \Omega)$ is asymptotically flat if the graph of $f$ is a  $C^2$ hypersurface up to boundary which satisfies the following conditions
\begin{itemize}
\item[(1)] Either $\lim_{|x| \to \infty} f(x) = C$ for some bounded constant $C$ or $\lim_{|x| \to \infty} f(x) = \infty$ (or $-\infty$); 
\item[(2)] $|D f(x)| = O(|x|^{-\frac{q}{2} })$ and $|D^2 f(x)| = O(|x|^{-\frac{q}{2}-1})$,
for some $q > \frac{n-2}{2}$, where $Df = (f_1,\ldots, f_n), D^2 f = (f_{ij})$ and $f_i = \p f/\p x^i$, $f_{ij} = \p^2 f/ \p x^i \p x^j$;
\item[(3)] The scalar curvature of the graph of $f$ is integrable over the graph of $f$.
\end{itemize}
\end{defi}
\begin{remark}
Under Condition (2), the induced metric of the graph of $f$ has the asymptotics
\[
	g_{ij} = \delta_{ij} + f_i f_j = \delta_{ij} + O(|x|^{-q}).
\] 
The decay rate $q$ is optimal in order for the ADM mass to be well-defined, assuming Condition (3) (see~\cite{Bartnik}).
\end{remark}
\begin{remark} \label{re:AF}
Condition (1) in Definition~\ref{de:AF} is not needed in the proof of Theorem~\ref{th:Penrose}. Condition (1) is actually redundant for $n\ge 6$ because by Condition (2) and using the mean value theorem along the radial direction and along the spherical direction, we have
\[
	\lim_{|x| \to \infty }f(x) = C,
\]
for some bounded constant $C$.
\end{remark}

\begin{defi}[\cite{Huang-Wu, Lam}]
Let $\Omega$ be a bounded subset in $\mathbb{R}^n$, $n\ge 3$, and let $f \in C^{2} (\mathbb{R}^n \setminus \Omega)$. The mass of the graph of $f$ is defined by
\begin{align*} \label{eq:mass} 
	m &= \frac{1}{2(n-1)  \omega_{n-1}}  \lim_{r\rightarrow \infty}\int_{S_r} \frac{1}{ 1 + |Df|^2 } \sum_{i,j}  (f_{ii} f_j - f_{ij} f_i)  \frac{x^j}{|x|}  \, d\sigma,
\end{align*}
where  $S_r = \{ (x^1, \dots, x^n): |x|=r\}$, and $d\sigma$ is the standard spherical volume measure of $S_r$. 
\end{defi}
\begin{remark}
The above definition of the mass is consistent with the classical definition of the ADM mass \cite[Lemma 5.8]{Huang-Wu}, \emph{cf}. \cite{Lam}.
\end{remark}

The spacelike $n$-dimensional Schwarzschild metric is a complete and conformally flat metric on $\mathbb{R}^n \setminus \{ 0 \}$
\[
	\left(\mathbb{R}^n \setminus \{ 0 \}, \left(1+ \frac{m}{2|x|^{n-2}} \right)^{\frac{4}{n-2}} \delta \right),
\]
where $m$ is the ADM mass. If $m\ge0$, the $n$-dimensional Schwarzschild solution can be isometrically embedded into Euclidean $\mathbb{R}^{n+1}$ as a smooth hypersurface. We refer the reader to \cite{Bray} for detailed discussions, especially for the $n=3$ case. We are interested in the region of the Schwarzschild solution outside its minimal $(n-1)$-hypersurface, which is graphical as shown in the following proposition. 
\begin{prop}\label{pr:Schwarzschild}
Denote by $B_r$ the open ball in $\mathbb{R}^n$ centered at the origin of radius $r$. The region of the Schwarzschild solution of mass $m> 0$ outside its minimal $(n-1)$-hypersurface can be represented as the graph of $h(x)$ over $\mathbb{R}^n \setminus B_{(2m)^{1/(n-2)}}$, where
\begin{align*}
	h(x) &= C_0 + \sqrt{ 8m (|x| - 2m)}  &&\mbox{if } n =3,\\
	h(x) &= C_0 + \sqrt{2m} \ln ( |x| + \sqrt{|x|^2 - 2m} ) &&\mbox{if } n =4,\\
	h(x) &= C_0 + O(|x|^{2-\frac{n}{2}}) \quad \mbox{for $|x| \gg 1$} &&\mbox{if } n \ge 5,
\end{align*}
for some constant $C_0$.
\end{prop}
\begin{proof}
Let $h \in C^2(\mathbb{R}^n\setminus B_{r_0})$ for some $r_0\ge0$ be rotationally symmetric. With a minor abuse of notation, we will write $h(x) = h(r)$ where $r= |x|$. By direct computation, the scalar curvature $R$ of the graph of $h$ is given by 
\[
   \frac{R}{2}= \frac{(n-1)h'' h'}{r\big[1 + (h')^2\big]^2} +  \frac{\binom{n-1}{2}(h')^2}{r^2\big[1 + (h')^2\big]},
\]
where $h' = \frac{dh}{dr}$ and $h'' = \frac{d^2 h} { dr^2}$.  Set  $y(r) = - \frac{1}{1+(h')^2}$.  Then $-1\le y\le 0$ and $y$ solves
\[	
	y' + \frac{n-2}{r} y + \frac{n-2}{r} = \frac{r R}{2(n-1)}.
\]
If $R\equiv 0$, for some constant $C_1\ge 0$, we have
\[
	y = C_1 r^{2-n} - 1.
\]
Therefore, for $r > (C_1)^{1/(n-2)}$,
\[
	(h')^2 = \frac{1}{1-C_1 r^{2-n}} - 1 = \frac{ C_1 r^{2-n} }{ 1- C_1 r^{2-n}}= \frac{C_1}{ r^{n-2}-C_1}.
\]
Then,
\[
	h(r)= \sqrt{C_1} \int  \frac{1}{\sqrt{r^{n-2} - C_1}} \, dr.
\]
Solving the integral, we have, for some constant $C_0$,
\begin{align*}
	h(r) &= C_0 + \sqrt{ 4 C_1 (r - C_1)} \qquad  &&\mbox{if } n =3,\\
	h(r) &= C_0 + \sqrt{C_1} \ln ( r + \sqrt{r^2- C_1} )\qquad &&\mbox{if } n =4,\\
	h(r) &= C_0 + O(r^{2-\frac{n}{2}}) \quad \mbox{for $r \gg 1$} \qquad &&\mbox{if } n \ge 5.
\end{align*}
By computing the mass directly, we have
\[
	m = \frac{1}{2(n-1)\omega_{n-1}}\lim_{r\to \infty} \int_{S_r} \frac{(n-1) (h')^2}{ r(1+ (h')^2)} \, d\sigma = \frac{C_1}{2}.
\]
It is straightforward to check that if $m> 0$, $h'(r) \to \infty$ as $r\to (2m)^{1/(n-2)}$, and the graph of $h$ over $\p B_{(2m)^{1/(n-2)}}$ is the minimal $(n-1)$-hypersurface in the graph of $h$.
\end{proof}

\begin{notation}
For a hypersurface, we denote by $A_{ij}$ the second fundamental form, by $A^i_j = \sum_k g^{ik} A_{kj}$ the shape operator where $g^{ik}$ is the inverse of the induced metric, by $H$ the mean curvature, and by $R$ the scalar curvature. If the hypersurface is the graph of $u$, we compute $A_{ij}$ with respect to the \emph{upward} unit normal vector and we can write the above quantities as the functions of $Du$ and $D^2 u$. We may also suppress the arguments when the context is clear. 
\begin{align*}
	g^{ik}(Du)  &= \left(\delta_{ik} - \frac{u_i u_k}{1+|Du|^2} \right)\\
	A_{ij} (Du, D^2 u )&= \frac{u_{ij}}{\sqrt{1+|Du|^2}}\\
	A^i_j (Du, D^2u) &=\sum_k \left(\delta_{ik} - \frac{u_i u_k}{1+|Du|^2} \right) \frac{u_{kj}}{\sqrt{1+|Du|^2}}\\
	H(Du, D^2 u) & = \sum_{i,j} \left(\delta_{ij} - \frac{u_i u_j}{1+|Du|^2} \right) \frac{u_{ij}}{\sqrt{1+|Du|^2}}\\
	R(Du, D^2 u) & = H^2(Du, D^2 u) - \sum_{i,j} A^i_j(Du, D^2 u) A^j_i(Du, D^2 u).
\end{align*}
\end{notation}

\begin{prop} \label{pr:strict-ellipticity}
Let $h$ be the graphing function of the Schwarzschild solution of $m>0$ in Proposition~\ref{pr:Schwarzschild}. Then the matrix $\big( H g^{ij} - \sum_k A^i_k g^{kj}\big)$ of the graph of $h$ is positive definite everywhere in $\mathbb{R}^n \setminus B_{(2m)^{1/(n-2)}}$.
\end{prop}
\begin{proof}
It suffices to show that $(H \delta^i_k - A^i_k)$ is positive definite, because $H g^{ij} - \sum_k A^i_k g^{kj} =\sum_k (H \delta^i_k - A^i_k) g^{kj} $ and $(g^{kj})$ is positive definite. By rotating coordinates, we can assume that $A^i_k = \mbox{diag}(\lambda_1, \dots, \lambda_n)$ where $\lambda_l$ are the principle curvatures. For the graph of a rotationally symmetric function $h(r)$, the principle curvatures are 
\[
	\frac{h''}{(1+(h')^2)^{3/2}}, \quad \mbox{and} \quad \frac{h'}{r\sqrt{1+(h')^2}}  \mbox{ with multiplicity $(n-1)$}.
\]
Therefore, the principle curvatures of the Schwarzschild solution are
\[
	-\frac{n-2}{2} \sqrt{2m} r^{-\frac{n}{2}}, \quad \mbox{and} \quad \sqrt{2m} r^{-\frac{n}{2}} \mbox{ with multiplicity $(n-1)$}.
\]
Hence, $(H \delta^i_k - A^i_k) \ge \frac{n-2}{2} \sqrt{2m} r^{-\frac{n}{2}} I$ where $I$ is the $n\times n$ identity matrix, so it is positive definite  everywhere in $\mathbb{R}^n \setminus B_{(2m)^{1/(n-2)}}$.
\end{proof}

The following two propositions were proven in our earlier paper \cite{Huang-Wu}. They play  important roles to prove the mean convexity of the asymptotically flat graphs and to derive the ellipticity of the linearized scalar curvature operator.

\begin{prop}[{\cite [Proposition 2.1]{Huang-Wu}}]\label{pr:id}
  Let $B = (b_{ij})$ be an $n \times n$ matrix with $n \ge 2$. Denote
\begin{align*}
    &\sigma_1(B) = \sum_{i=1}^n b_{ii}, \quad \sigma_1(B|k) = \bigg(\sum_{i=1}^n b_{ii}\bigg) - b_{kk},\\
    &\sigma_2(B) = \sum_{1 \le i < j \le n} (b_{ii} b_{jj} - b_{ij}b_{ji}).
\end{align*}
  For each $1 \le k \le n$, we have
  \begin{equation*} \label{eq:id1}
    \begin{split}
    \sigma_1(B) \sigma_1(B|k) 
    & = \sigma_2(B) + \frac{n}{2(n-1)} \left(\sigma_1(B|k)\right)^2 + \sum_{1 \le i < j \le n} b_{ij} b_{ji} \\
    & \quad + \frac{1}{2(n-1)} \sum_{1 \le i < j \le n \atop i\neq k, j\neq k } (b_{ii} - b_{jj})^2,
    \end{split}
  \end{equation*}
  where the last term is zero when $n = 2$.
  In particular, if $B$ is real and $b_{ij}b_{ji} \ge 0$ for all $1 \le i< j \le n$, then
  \[
     \sigma_1(B) \sigma_1(B|k) \ge \sigma_2(B) + \frac{n}{2(n-1)} \left( \sigma_1(B|k) \right)^2
  \]
  with equality if and only if $b_{ii}$ are equal for all $i = 1, \dots, n$ and $i\neq k$, and $b_{ij} b_{ji} = 0$ for all $i,j = 1,\dots, n$ and $i \ne j$.
\end{prop}

\begin{notation}
Let $N$ be a (piece of) hypersurface in Euclidean space, and  let $\mu$ be a unit normal vector field to $N$. The mean curvature of $N$ defined by $\mu$ is given by 
\[
	H_N = - \mbox{div}_0 \mu,
\] 
where $\mbox{div}_0$ is the Euclidean divergence operator. (The $n$-dimensional sphere of radius $r$ has mean curvature $n/r$ with respect to the inward unit normal vector by this convention.) We denote by $\langle \cdot, \cdot \rangle$ the standard metric on Euclidean space.  With a slight abuse of notation, we may view $\eta$ as a vector in $\mathbb{R}^n$, as well as a vector in $\mathbb{R}^{n+1}$ by letting the last component be zero.
\end{notation} 
\begin{prop}[{\cite[Theorem 2.2]{Huang-Wu}}]\label{pr:HHR}
Let $M$ be a $C^{2}$ hypersurface in $\mathbb{R}^{n+1}$. Consider the height function $h : M \rightarrow \mathbb{R}$ given by $ h(x^1, \dots, x^{n+1})= x^{n+1}$. Let $a$ be a regular value of $h$. Denote by
\[
	\Sigma = M \cap \{ x^{n+1} = a\},
\] 
which is a $C^{2}$ hypersurface in $\{ x^{n+1} = a\}$ and $|\nabla^M h| > 0$ at every point in $\Sigma$. Denote by $\nu$ and $\eta$ the unit normal vector fields to $M \subset \mathbb{R}^{n+1}$ and $\Sigma \subset \{ x^{n+1} = a\}$, respectively, and denote by $H$ and $H_{\Sigma}$ the mean curvatures of $M\subset \mathbb{R}^{n+1}$ and $\Sigma \subset \{ x^{n+1} = a\}$ defined by $\nu$ and $\eta$, respectively. Let $R$ be the induced scalar curvature of $M$. Then, at every point of $\Sigma$, 
\begin{equation*} \label{eq:HHR}
	  \langle \nu, \eta \rangle H H_{\Sigma} \ge \frac{R}{2} + \frac{n}{2(n-1)} \langle \nu, \eta \rangle^2 H_{\Sigma}^2
\end{equation*}
with the equality at a point in $\Sigma$ if and only if 
$(M, \Sigma)$ satisfies the following two conditions at the point:
\begin{enumeratei}
  \item \label{it:S} $\Sigma \subset \mathbb{R}^n$ is umbilic, with the principal curvature $\kappa$;
  \item \label{it:M} $M \subset \mathbb{R}^{n+1}$ has at most two distinct principal curvatures, and one of them is equal to $\langle \nu, \eta \rangle \kappa$, with multiplicity at least $n-1$.
\end{enumeratei}
\end{prop}

\section{Proof of Theorem~\ref{th:mean-convexity}}

\begin{notation}
Let $M$ be a hypersurface in Euclidean space and let $\textup{int}(M)$ be the set of \emph{interior} points in $M$, i.e., $\textup{int}(M) = M \setminus \p M$. The set of \emph{interior geodesic points} is given by
\begin{equation} \label{eq:defM0}
	M_0 = \{ p \in \textup{int}(M): (A^i_j)= 0 \mbox{ at } p\}.
\end{equation}
\end{notation}

A classical result of R. Sacksteder~\cite[Lemma 6]{Sacksteder} characterizes the set of geodesic points. While he proved the statement for complete hypersurfaces, the statement can be easily generalized to hypersurfaces with boundary,  see also \cite[Lemma 3.6]{Huang-Wu} and \cite[Lemma 4.5]{Huang-Wu-sphere}. 
\begin{lemma}\label{le:geodesic}
Let $M$ be a $C^{n+1}$ hypersurface in $\mathbb{R}^{n+1}$, and let $M_0'$ be a connected component of $M_0$. Then $M_0'$ lies in a hyperplane which is tangent to $M$ at every point in $M_0'$.
\end{lemma}

To prove Theorem~\ref{th:mean-convexity}, let us recall the following results in~\cite{Huang-Wu}. 


\begin{lemma}[{\cite[Proposition 3.1]{Huang-Wu}}] \label{le:nonempty}
Let $W$ be an open subset in $\mathbb{R}^n$, not necessarily bounded. Let $p\in \p W$, and denote by $B(p)$ an open ball in $\mathbb{R}^n$ centered at $p$. Suppose $f \in C^2(  W\cap B(p)) \cap C^1( \overline{W} \cap B(p))$ satisfies 
\begin{align*}
	& H(Df, D^2 f) \ge 0\quad &&\mbox{in } W  \cap B(p)   \\
	&f= c, \; |Df|=0\quad &&\mbox{on } \p W  \cap B(p),
\end{align*}	
for some constant $c$. Then either $f\equiv c$ in $W  \cap B(p)$, or
\[
	\{ x \in  W  \cap B(p): f(x) > c \} \neq \emptyset.
\]	
\end{lemma}

\begin{thm}[{\cite[Theorem 3.9]{Huang-Wu}}] \label{th:zero}
Let $W$ be a bounded open subset in $\mathbb{R}^n$ and let $N$ be an open neighborhood of $\partial W$.  If $f\in C^{n+1}(\overline{W} \cap N)$, $f=c, |Df|=0, |D^2 f| = 0$ on $\partial W$, and the scalar curvature of the graph of $f$ is non-negative, then $f\equiv c$ on $W\cap N$.
\end{thm}

 \begin{thm} \label{th:unbounded}
Let $f\in C^{n+1} (\mathbb{R}^n \setminus \overline{\Omega}) \cap C^0 (\mathbb{R}^n \setminus \Omega)$ and $|Df(x)| \to \infty$ as $x\to \partial \Omega$. Denote by $M$ the graph of $f$.  Suppose $M$ is $C^{n+1}$ up to boundary and has non-negative scalar curvature. Suppose that the mean curvature $H$ of $M$ changes signs. Let $M_+$ be a connected component of $\{p \in M: H \ge 0 \mbox{ at } p\}$ that contains a point of positive mean curvature. Then both $M_+$ and the boundary of each component of $M\setminus M_+$, except $\partial M$, must be unbounded. 

\end{thm} 

\begin{proof}

Note that $M$ is homeomorphic to $\mathbb{R}^n \setminus \Omega$ given by the projection $\pi(x, f(x)) = x$. Let $\mathbb{R}^n \setminus \pi(M_+) = \sqcup_{\alpha} U_{\alpha}$ where each $U_{\alpha}$ is a connected component and $\Omega \subset U_{\alpha_0}$ for some $\alpha_0$. Because $\pi(M_+)$ is connected, each $\partial U_{\alpha}$ is connected~\cite[Proposition A.3]{Huang-Wu}. We claim that $U_{\alpha}$ is unbounded for each $\alpha \neq \alpha_0$.  Suppose to the contrary that $U_{\alpha}$ is bounded for some $\alpha \neq \alpha_0$. Denote $\pi^{-1} (\partial U_{\alpha}) = \Gamma$. By Lemma~\ref{le:geodesic}, $\Gamma \subset M_0$ lies in a hyperplane $\Pi$. Note that $\overline{\Gamma}$ does not intersect with $\p M$ because $|Df| = \infty$ on $\partial M$. Hence $M$ can be represented as the graph of a $C^{n+1}$-function $u$ in an open neighborhood of $\Gamma$ in $\Pi$, and $u$ satisfies $u=0, |Du|=0, |D^2 u|=0$ on $\Gamma$.  By Theorem~\ref{th:zero}, we have $u\equiv 0$ on $\pi^{-1}(U_{\alpha})$ and hence $\pi^{-1}(U_{\alpha}) \subset M_+$. It contradicts. Similarly, one can show that $U_{\alpha_0}$ is unbounded and $\partial U_{\alpha_0}$ does not intersect $\partial \Omega$, unless $U_{\alpha_0} = \Omega$.

Furthermore, $M_+$ must be unbounded. Suppose not. Then the complement $\mathbb{R}^n \setminus \pi(M_+) $ has a unique unbounded connected component that contains infinity. Therefore, either $\mathbb{R}^n \setminus \pi(M_+) = \Omega \sqcup U$ where $U$ is the unbounded connected component that contains infinity, or $\mathbb{R}^n \setminus \pi(M_+) = U_{\alpha_0}$ if $U_{\alpha_0}$ is unbounded. Hence we have either $\partial \pi(M_+) = \partial \Omega \sqcup \partial U$ or $\partial \pi(M_+)= \partial U_{\alpha_0}$. Applying  Theorem~\ref{th:zero} to $\pi(M_+)$ on either $\partial U$ or $\partial U_{\alpha_0}$, we have $H\equiv 0$ on $M_+$ which contradicts that $M_+$ contains a point of positive mean curvature.

\end{proof}
\begin{pf-mean-convexity}
Suppose to the contrary that $H$ changes signs. Let $M_+$ be a connected component of $\{p \in M: H \ge 0 \mbox{ at } p\}$ that contains a point of positive mean curvature. By Theorem~\ref{th:unbounded},  a component $\Gamma$ of $\partial M_+$ is unbounded. By Lemma~\ref{le:geodesic}, $\Gamma$ lies in a hyperplane $\Pi$ and $M$ is tangent to $\Pi$ at $\Gamma$. By the assumption that $f$ is asymptotically flat, the upward unit normal vector of $M$ converges to $\p_{n+1}$ at infinity. Because $M$ is tangent to the hyperplane $\Pi$ at an unbounded set,  we must have $\lim_{|x|\to \infty} f(x) = C$ for some bounded constant $C$ and  $\Pi = \{ x^{n+1} = C\}$.  

By Lemma~\ref{le:nonempty}, the level set $ \{ x : f(x)= C+\epsilon\}$ has non-empty intersection with $\{ p\in \textup{int}(M): H > 0 \mbox{ at } p \}$ for all $\epsilon>0$ sufficiently small. Let $\Sigma_{C+\epsilon}$ be a connected component of the level set which  intersects  $\{ p\in \textup{int}(M): H > 0 \mbox{ at } p \}$.  Note that $\Sigma_{C+\epsilon}$ is closed if $\epsilon \neq 0$, and that $H\ge 0$ at every point of $\Sigma_{C+\epsilon}$ because by Theorem~\ref{th:unbounded} the mean curvature $H$ can only change signs through an unbounded subset of $M_0$, which must lie on $\{ x^{n+1} = C\}$. By Morse-Sard theorem, $\Sigma_{C+\epsilon}$ is a $C^{n+1}$ submanifold with $|D f| > 0$ for almost every $\epsilon$. Let $\eta =D f / |D f|$ be a unit normal vector of $\Sigma_{C+\epsilon}$ where $C+\epsilon$ is a regular value. For $\epsilon>0$ sufficiently small, $\eta$ points \emph{inward} to the region enclosed by $\Sigma_{C+\epsilon}$ because $f$ decreases to $C$ at infinity. Let $H_{\Sigma_{C+\epsilon}}$ be the mean curvature with respect to $\eta$. By Proposition~\ref{pr:HHR}, $H_{\Sigma_{C+\epsilon}} \le 0$ at every point of $\Sigma_{C+\epsilon}$, which contradicts  compactness of $\Sigma_{C+\epsilon}$.
\qed

\end{pf-mean-convexity}

\section{Ellipticity and maximum principles}  \label{se:maximum}
In this section, we will derive various maximum principles for graphs with non-negative scalar curvature. The scalar curvature equation of the graphing function is fully nonlinear. Its linearization, first introduced by~\cite{Reilly:1973}, gives a second-order differential equation. The linearized equation may not be elliptic in general. It is known that if the scalar curvature is a \emph{positive} constant, then the strict ellipticity trivially holds pointwise (see, for example,~\cite{Rosenberg:1993}). However, this is no longer true if the scalar curvature may vanish. An important step to establish our maximum principles is to explore the (strict) ellipticity of the linearized scalar curvature equation. 

\begin{lemma} \label{le:linearized}
Let $u$ be a $C^2$ function and let $R$ be the scalar curvature of the graph of $u$. Then
\[
	\frac{\p R}{ \p u_{ij}} (Du, D^2 u)=\frac{2}{\sqrt{1+|Du|^2}} \left( H g^{ij} - \sum_k A^i_k g^{kj}\right).
\]
\end{lemma}
\begin{proof}
 By chain rule, 
   \begin{align*}
     \frac{\partial R}{\partial u_{ij}}
     & = \sum_{k,l} \frac{\partial R}{\partial A^k_l} \frac{\partial A^k_l}{\partial u_{ij}} \\
  &   = \sum_{k,l} \frac{\partial R}{\partial A^k_l} \frac{\partial }{\partial u_{ij}}\left(\sum_{p} g^{kp} A_{pl} \right) \\
   &  = \sum_{k} \frac{\partial R}{\partial A^k_i} \frac{g^{kj}}{\sqrt{1+|Du|^2}}\\
     & =  2 H\frac{g^{ij}}{\sqrt{1+|Du|^2}}-2 \sum_{k} A^i_k \frac{g^{kj}}{\sqrt{1+|Du|^2}}.
   \end{align*}
\end{proof}

\begin{prop} \label{pr:ellipticity}
Let $u$ be a $C^2$ function, and let $R$ and $H$ be the scalar curvature and mean curvature of the graph of $u$, respectively. If $R\ge 0$ and $H\ge 0$, then the matrix $\big( H g^{ij} - \sum_k A^i_k g^{kj}\big)$ is semi-positive definite.
\end{prop}
\begin{proof}
Because $H g^{ij} - \sum_k A^i_k g^{kj} =\sum_k (H \delta^i_k - A^i_k) g^{kj} $ and $(g^{kj})$ is positive definite, it suffices to  prove that $(H \delta^i_k - A^i_k)$ is semi-positive definite. By rotating the coordinates, we assume $(A^i_k) = \mbox{diag}(\lambda_1, \dots, \lambda_n)$. Then
\[
	(H \delta^i_k - A^i_k) = \mbox{diag} (\sigma_1 (A|1), \dots, \sigma_1(A|n)).
\]
By Proposition~\ref{pr:id}, because $H = \sigma_1(A) \ge 0$ and $R = 2\sigma_2(A)\ge 0$, we have $\sigma_1(A|k)\ge 0$ for all $k=1,\dots, n$.
\end{proof}

\begin{thm}[Strong maximum principle for the interior point] \label{th:strong}
Let $\Omega$ be a connected open subset in $\mathbb{R}^n$. Suppose $u, v \in C^2 (\Omega)$, $u\ge v$ in $\Omega$, and $u,v$ satisfy 
\begin{align*}
	& R(Du, D^2u) = 0,\quad R(Dv, D^2 v) \ge 0,\\
	& H(Du, D^2u) \ge 0, \quad \mbox{and} \quad H(Dv, D^2 v) \ge 0 \quad \mbox{ in } \Omega.
\end{align*}
We assume that either $u$ or $v$ satisfies $\left(H g^{ij} - \sum_k A^i_k g^{kj} \right)$ being positive definite in $\Omega$. If $u=v$ at some point in $\Omega$, then $u\equiv v$ in $\Omega$.
\end{thm}
\begin{proof} 
Let $R(p, \xi) \in C^1(\mathbb{R}^n \times \mathbb{R}^{n\times n})$ be the scalar curvature operator. Hence
\begin{align*}
	0&\ge R(Du, D^2 u) - R (Dv, D^2 v) \\
	&=R(Du, D^2 u) - R(Du, D^2 v) + R(Du, D^2 v) - R(Dv, D^2 v)\\
	&=\sum_{i,j} a^{ij}(u_{ij}-v_{ij})  +  \sum_i b^i\, (u_i-v_i),
\end{align*}
where 
\[
	b^i =\int_0^1 \frac{ \partial R}{\partial p_i} (tDu + (1-t ) Dv, D^2 v) \, dt 
\]	
and by Lemma~\ref{le:linearized}
\begin{align*}
	a^{ij}&= \int_0^1 \frac{\partial R}{\partial \xi_{ij}} (Du, tD^2 u + (1-t) D^2 v) \, dt\\
	&=\frac{1}{\sqrt{1+|Du|^2}} \bigg( \big(H(Du, D^2 u) g^{ij}(Du) - \sum_k A^i_k(Du, D^2 u) g^{kj}(Du) \big)\\
	&\quad+\big(  H(Du, D^2 v) g^{ij}(Du)- \sum_k A^i_k(Du, D^2 v)) g^{kj}(Du) \big) \bigg).
\end{align*}
If $u=v$ at $p \in \Omega$, then $Du = Dv$ at $p$. Then, by the assumption and Proposition~\ref{pr:ellipticity}, $(a^{ij})$ is positive definite at $p$. By continuity, $(a^{ij})$ is positive definite in an open neighborhood $\Omega'$ of $p$ in $\Omega$. Then by the standard strong maximum principle, $u\equiv v$ in $\Omega'$. Hence, the set $\{p\in \Omega: u(p) = v(p) \}$ is open and closed. Because $\Omega$ is connected, we prove that $u\equiv v$ in $\Omega$.
\end{proof}

\begin{thm}[Strong maximum principle for the boundary point] \label{th:boundary}
Let $\Omega_1, \Omega_2$ be connected open sets in $\mathbb{R}^n$ such that $\Omega_1 \subset \Omega_2$. Suppose  $p \in \p \Omega_1 \cap \p \Omega_2 \neq \emptyset$ and  $\p \Omega_1, \p \Omega_2$ are $C^1$ near $p$.

Let $u \in C^2 (\Omega_1)\cap C^0 (\overline{\Omega}_1), v\in C^2 (\Omega_2) \cap C^0 (\overline{\Omega}_2)$.  Suppose the graphs of $u, v$ are $C^2$ hypersurfaces up to boundary which satisfy
\begin{align*}
	& R(Du, D^2u) = 0,\quad R(Dv, D^2 v) \ge 0 \\
	& H(Du, D^2 u) \ge 0,  \quad \mbox{and} \quad H(Dv, D^2 v) \ge 0 \quad\mbox{ in }  \Omega_1. 
\end{align*}
We also assume that either $u$ or $v$ satisfies $\left(H g^{ij} - \sum_k A^i_k g^{kj} \right)$ being positive definite in $ \Omega_1$. If $u\ge v\ge 0$ in $ \Omega_1$  and $u|_{\p \Omega_1 \cap B_r(p)} = v|_{\p \Omega_2\cap B_r(p)} = 0$ for some open ball centered at $p$ of radius $r$ with $|Du(x)|\to \infty$ and $|Dv(x)|\to \infty$ as $x \to p \in \p \Omega_1 \cap \p \Omega_2$, then $u\equiv v$ in $\Omega_1$.
\end{thm}
\begin{proof}
Let $\Pi$ be the vertical hyperplane in $\mathbb{R}^{n+1}$ so that $\Pi$ is tangent to $\p \Omega_1 \times\{ x^{n+1}\mbox{-axis}\}$ at $p \times \{ x^{n+1}\mbox{-axis}\}$. The graphs of $u, v$ near $p$ can be locally represented as the graphs of some functions $\tilde{u}, \tilde{v}$ over a subset of $\Pi$, say $\tilde{u}, \tilde{v} \in C^2 (D \times [0,\epsilon])$ where $p\in \mbox{int}(D)\subset\{ x^{n+1} = 0\}$ and $D \times [0,\epsilon] \subset \Pi$. Moreover, $\tilde{u}$ and $\tilde{v}$ satisfy $\tilde{u} \ge \tilde{v} $ in $D \times [0,\epsilon]$,  $\tilde{u}=\tilde{v}$ and $\p_{n+1}\tilde{u} = \p_{n+1}\tilde{v}=0$ at $p \in D\times \{0\}$, and 
\begin{align*}
	& R(D\tilde{u}, D^2\tilde{u}) = 0,\quad R(D\tilde{v}, D^2 \tilde{v}) \ge 0 \\
	& H(D\tilde{u}, D^2 \tilde{u}) \ge 0,  \quad \mbox{and} \quad H(D\tilde{v}, D^2 \tilde{v}) \ge 0 \quad\mbox{ in } D \times [0,\epsilon].
\end{align*}
Either $\tilde{u}$ or $\tilde{v}$ has $\left(H g^{ij} - \sum_k A^i_k g^{kj} \right)$ being positive definite in $D\times [0,\epsilon]$. As analyzed in the proof of Theorem~\ref{th:strong}, $(\tilde{u}-\tilde{v})$ satisfies
\[
	0\ge \sum_{i,j} a^{ij} (\tilde{u}_{ij} - \tilde{v}_{ij}) + \sum_i b^i  (\tilde{u}_i - \tilde{v}_i), 
\]
where $(a^{ij})$ is positive definite in $D\times [0,\epsilon]$ with  a possibly smaller $\epsilon$. Then by the standard Hopf boundary point lemma, $\tilde{u} = \tilde{v}$ at some interior points of $D \times [0,\epsilon]$. Hence $u=v$ at some interior points in $\Omega_1$, and by Theorem~\ref{th:strong}, $u \equiv v$ everywhere in  $\Omega_1$. 
\end{proof}

To prove Theorem~\ref{th:Schwarzschild}, we need the following version of the strong maximum principle for the boundary point, where the domains of $u$ and $v$ are complement to each other. The proof is nearly identical to the proof of Theorem~\ref{th:boundary}, so we omit it.

\begin{thm} \label{th:strong-max}
Let $\Omega$ be an open subset in $\mathbb{R}^n$. Let $p \in \p \Omega$ and  consider the open ball $B_r(p)$ centered at $p$ of radius $r$ for some $r >0$ small. Suppose $\p \Omega \cap B_r (p)$ is $C^1$. Let $u \in C^2 (B_r(p) \setminus (\overline{\Omega \cap B_r(p)}))\cap C^0 (B_r(p) \setminus (\Omega\cap B_r(p))), v\in C^2 (\Omega\cap B_r(p)) \cap C^0 (\overline{\Omega \cap B_r(p)})$ and $u\ge 0, v\le 0$.  Suppose  the graphs of $u, v$ are $C^2$ hypersurfaces up to boundary which satisfy
\begin{align*}
	& R(Du, D^2u) = 0,\quad R(Dv, D^2 v) = 0 \\
	& H(Du, D^2 u) \ge 0,  \quad \mbox{and} \quad H(Dv, D^2 v) \ge 0. 
\end{align*}
We also assume that the matrix $\left(H g^{ij} - \sum_k A^i_k g^{kj} \right)$ of $u$ is positive definite. If $u|_{\p \Omega \cap B_r(p)} = v|_{\p \Omega\cap B_r(p)} = 0$. then   $|Du(x)|$ and $|Dv(x)|$  cannot both tend to $\infty$ as $x \to p \in \p \Omega$.
\end{thm}

\begin{thm}[Global strong maximum principle] \label{th:global-maximum}
Let $\Omega$ be a bounded subset  (not necessarily connected) in $\mathbb{R}^n$, and let $v\in C^2 (\mathbb{R}^n \setminus  \Omega)$ be asymptotically flat. We assume that the graph of $v$ satisfies $R= 0$ and $H\ge 0$ in $\mathbb{R}^n \setminus \Omega$. Let $h$ be the Schwarzschild solution given by Proposition~\ref{pr:Schwarzschild}. Then there exists $r\gg 1$ so that, for any $r_2 > r_1 \ge r$, 
\begin{align*}
	 \max_{\overline{B}_{r_2} \setminus B_{r_1}} (h-v)&= \max_{S_{r_2} \cup S_{r_1}} (h-v)\\
	 \min_{\overline{B}_{r_2} \setminus B_{r_1}} (h-v)&= \min_{S_{r_2} \cup S_{r_1}} (h-v).
\end{align*}
If $(h-v)$ attains its maximum or minimum at an interior point in $B_{r_2}\setminus \overline{B}_{r_1}$, then $(h-v)$ must be a constant in $\overline{B}_{r_2} \setminus B_{r_1}$.
\end{thm}
\begin{proof}
As computed in the proof of Theorem~\ref{th:strong}, we have
\begin{align*}
	0&= R(Dh, D^2 h) - R (Dv, D^2 v)\\
	&=\sum_{i,j} a^{ij} (h_{ij}-v_{ij})  +  \sum_i b^i\, (h_i-v_i).
\end{align*}
where
\begin{align*}
	a^{ij}&=\frac{1}{\sqrt{1+|Dh|^2}} \sum_k \bigg( H(Dh, D^2 h) \delta^i_k-  A^i_k(Dh, D^2 h)\\
	&\quad+  H(Dh, D^2 v) \delta^i_k-  A^i_k(Dh, D^2 v)   \bigg) g^{kj}(Dh).
\end{align*}
We shall prove that $(a_{ij})$ is positive definite in $\mathbb{R}^n \setminus B_r$ for $r\gg1$. Then the lemma follows directly from the standard maximum principles. Because $(g^{kj})$ is positive definite, we can prove the positivity of $(a^{ij})$ by showing that the matrix
\begin{align} \label{eq:matrix}
	 (H(Dh, D^2 h) \delta^i_k-  A^i_k(Dh, D^2 h)+  H(Dh, D^2 v) \delta^i_k -  A^i_k(Dh, D^2 v) )
\end{align}
is positive definite. By direct computation, 
\begin{align*}
	& H(Dh, D^2 v) \delta^i_k -  A^i_k(Dh, D^2 v) \\
	 &=H(Dv, D^2 v) \delta^i_k -  A^i_k(Dv, D^2 v) + O(|Dv|^2 |D^2 v| + |Dh|^2 |D^2 h|)\\
	 &=H(Dv, D^2 v) \delta^i_k -  A^i_k(Dv, D^2 v) + o\left(r^{(-3n+2)/4}\right),
\end{align*}
where $r = |x|$ and we use the asymptotic flatness of $h$ and $v$. By Proposition~\ref{pr:ellipticity}, $\left(H(Dv, D^2 v) \delta^i_k -  A^i_k(Dv, D^2 v)\right)$ is semi-positive definite.  By Proposition~\ref{pr:strict-ellipticity}, 
\[
	\left(H (Dh, D^2 h)\delta^i_k - A^i_k(Dh, D^2 h)\right) \ge \frac{n-2}{2} \sqrt{2m} r^{-n/2} I.
\]
The right hand side above is positive enough to absorb the error term $o(r^{(-3n+2)/4})$ if $r\gg 1$. Hence \eqref{eq:matrix} is positive definite. 
\end{proof}

\section{Proofs of Theorem~\ref{th:equality} and Theorem~\ref{th:Schwarzschild}} \label{se:main}

We need the the following inequalities. Proposition~\ref{pr:Alexandrov-Fenchel} is a special case of the Alexandrov-Fenchel inequalities. The classical result was proven for a convex domain. It has been generalized to a star-shaped domain $\Omega$ with a mean-convex boundary~\cite{Guan-Li}  or a domain whose boundary is outer-minizing by Huisken and  by \cite{Freire-Schwartz}. 
\begin{prop}\label{pr:Alexandrov-Fenchel}
Let $\Omega \subset \mathbb{R}^n$ be star-shaped or outer-minimizing and let $\Sigma = \p \Omega$ be mean convex. Denote by $H_{\Sigma}$ the mean curvature of $\Sigma$ with respect to the inward unit normal vector. Then 
\[
	\frac{1}{2(n-1) \omega_{n-1} } \int_{\Sigma} H_{\Sigma} \, d\sigma \ge \frac{1}{2} \left(\frac{|\Sigma| }{ \omega_{n-1}} \right)^{\frac{n-2}{n-1}}
\]
with equality if and only if $\Sigma$ is an $(n-1)$-sphere. 
\end{prop}

\begin{prop} \label{pr:elementary}
Let $a_1, a_2, \dots, a_k$ be non-negative real numbers and $0\le \beta \le 1$. Then 
\[
	\sum_{i=1}^k a_i^{\beta} \ge \bigg(\sum_{i=1}^k a_i \bigg)^{\beta}.
\]
If  $0\le \beta < 1$, the equality holds if and only if at most one of $a_i$ is non-zero.
\end{prop}
\begin{proof}
Without loss of generality, we assume $k=2$. We shall prove that $x^{\beta} + y^{\beta} \ge (x+y)^{\beta}$ if $0\le \beta \le 1$ and $x, y \ge 0$. Fix $\beta, x$ and define $w(y) = x^{\beta} + y^{\beta} - (x+y)^{\beta}$. Then $w(0)=0$ and 
\[
	w'(y) = \beta \left( y^{\beta-1} - (x+y)^{\beta-1}\right) \ge 0 \quad \mbox{ for all } y\ge 0.
\]	
Hence $w(y)\ge 0$ for all $y\ge 0$ with $w(y) =0$ if and only if $x=0$, or $y=0$, or $\beta=1$.
\end{proof}

\begin{pfPenrose}[\cite{Lam}]
The scalar curvature of the graph of $f$ has a divergence form in terms of the derivatives of $f$ \cite{Reilly:1973-Mich} (see also \cite{Lam, Huang-Wu})
\[
    R=  \sum_j \partial_j \sum_i \left(\frac{f_{ii} f_j - f_{ij} f_i}{1+|Df|^2} \right). 
\]
Let $\Omega^{\epsilon}$ be a bounded open set in $\mathbb{R}^n$ that contains $\Omega$ with $\Omega^{\epsilon} \to \Omega$ as $\epsilon \to 0$, and let each connected component $\Sigma^{\epsilon}_k$ of $\p \Omega^{\epsilon}$ be the level set of $f$. By applying the divergence theorem over $\mathbb{R}^n \setminus \Omega^{\epsilon}$,  we have
\begin{align*} 
	&2(n-1) \omega_{n-1} m\\
	&=\lim_{r\to \infty}\int_{S_r} \frac{1}{ 1 + |Df|^2 } \sum_{i,j}  (f_{ii} f_j - f_{ij} f_i)  \frac{x^j}{|x|}  \, d\sigma \\
	&=  \int_{\mathbb{R}^n \setminus \Omega^{\epsilon}} R \, dx-\sum_k \int_{\Sigma^{\epsilon}_k} \frac{1}{ 1 + |Df|^2 } \sum_{i,j}  (f_{ii} f_j - f_{ij} f_i) \eta^j  \, d\sigma \\
	&= \int_{\mathbb{R}^n \setminus \Omega^{\epsilon}} R \, dx+\sum_k \int_{\Sigma^{\epsilon}_k} \frac{|Df|^2}{ 1 + |Df|^2} H_{\Sigma^{\epsilon}_k}\, d\sigma, 
\end{align*}
where $H_{\Sigma_k^{\epsilon}}$ denotes the mean curvature of the level set $\Sigma_k^{\epsilon}$ with respect to the unit normal vector $\eta$ pointing \emph{inward} to the region enclosed by $\Sigma_k^{\epsilon}$ (\emph{cf.} \cite[Proof of Lemma 5.6]{Huang-Wu}). Let $\epsilon$ tend to zero. Then each level set $\Sigma^{\epsilon}_k$ tends to  the connected component $\Sigma_k$ of $\p \Omega$ and $|Df| \to \infty$. Therefore, we have
\begin{align*}
	m&= \frac{1}{2(n-1) \omega_{n-1} }\left( \int_{\mathbb{R}^n \setminus \Omega} R \, dx+\sum_k \int_{\Sigma_k} H_{\Sigma_k}\, d\sigma \right)\notag
\\
	&\ge \sum_k \frac{1}{2(n-1) \omega_{n-1} } \int_{\Sigma_k}  H_{\Sigma_k}\, d\sigma\\
	&\ge   \sum_k \frac{1}{2} \left(\frac{|\Sigma_k| }{ \omega_{n-1}} \right)^{\frac{n-2}{n-1}} \qquad \mbox{(by Proposition~\ref{pr:Alexandrov-Fenchel})}\\
	&\ge \frac{1}{2} \left( \frac{|\p \Omega|}{\omega_{n-1}}\right)^{\frac{n-2}{n-1}} \qquad \mbox{(by Proposition~\ref{pr:elementary})}.
\end{align*}
\qed
\end{pfPenrose}

\begin{cor} \label{cr:equality}
Under the conditions of Theorem~\ref{th:Penrose}, 
if 
\[
	m= \frac{1}{2} \left( \frac{|\p \Omega|}{\omega_{n-1}}\right)^{\frac{n-2}{n-1}},
\] 
then $R\equiv 0$ in $\mathbb{R}^n \setminus \Omega$, and $\p \Omega$ is a sphere of radius $(2m)^{\frac{1}{n-2}}$. 
\end{cor}

The proof of Theorem~\ref{th:equality} makes use the maximum principles proven in Section~\ref{se:maximum}. The case $n=3$ or $4$ is more subtle because the graphing function of the Schwarzschild solution tends to infinity as $|x| \to \infty$, and difficulty arises when comparing two unbounded graphs.  In that case, instead of using the strong maximum principles Theorem~\ref{th:strong} or Theorem~\ref{th:boundary} directly, we first control the growth at infinity, by the asymptotic flatness of the graph and Theorem~\ref{th:global-maximum}.  

\begin{pfthm}
Suppose  the equality of the Penrose inequality holds. By Corollary~\ref{cr:equality}, the scalar curvature of the graph of $f$ is identically zero everywhere and $\p \Omega$ is a round sphere of radius $(2m)^{\frac{1}{n-2}}$. By translating $f$, we assume that $f=0$ on $\p \Omega$  and $\p \Omega = S_{(2m)^{\frac{1}{n-2}}} \subset \{ x^{n+1}=0\}$. By Theorem~\ref{th:mean-convexity}, the mean curvature of the graph of $f$ must have a sign. Suppose that $H\ge 0$ with respect to the upward unit normal. 
Then, we have either $\lim_{|x| \to \infty} f(x) = C$ for some positive constant $C$ or $\lim_{|x| \to \infty} f(x) = +\infty$.

Let $h$ be the function in Proposition~\ref{pr:Schwarzschild} which gives the exterior region of the Schwarzschild solution of mass $m$ outside its minimal boundary. By translating $h$, we assume that its minimal boundary is $S_{(2m)^{\frac{1}{n-2}}} \subset \{ x^{n+1}=0\}$. We consider two cases, depending on the dimension $n$.

\textbf{Case 1:}  $n\ge 5$. In this case, $\lim_{|x| \to \infty} h(x)= C_0$ for some bounded constant $C_0$.   

If $\lim_{|x|\to \infty} f(x) = C\le C_0$, we let $u_{\lambda}=h+ \lambda$ for $\lambda\ge0$. For $\lambda$ sufficiently large, $u_{\lambda} > f$. We then continuously decrease $\lambda$, until $u_{\lambda} = f$ at $p\in \mathbb{R}^n \setminus B_{(2m)^{\frac{1}{n-2}}}$ for the first time. If $p$ is an interior point,  $u_{\lambda} \equiv f$ by Theorem~\ref{th:strong}. If $p$ is a boundary point in $S_{(2m)^{\frac{1}{n-2}}}$, $u_{\lambda} \equiv f$ by Theorem~\ref{th:boundary}. Hence, the graph of $f$ is identical to the exterior region of the Schwarzschild solution of mass $m$ outside its minimal boundary.

If $\lim_{|x|\to \infty} f(x) = C\ge C_0$ or $\lim_{|x| \to \infty} f(x) = +\infty$, we consider $v_{\lambda}=h-\lambda$ for $\lambda\ge 0$. Note that $f > v_{\lambda}$ for $\lambda$ sufficiently large. We then continuously decrease $\lambda$ until $f=v_{\lambda}$ for the first time. Then by either Theorem~\ref{th:strong} or Theorem~\ref{th:boundary}, we have $f\equiv h$ in $\mathbb{R}^n \setminus B_{(2m)^{\frac{1}{n-2}}}$.

\textbf{Case 2:}  $ n= 3$ or $4$. (In this case, $\lim_{|x| \to \infty} h(x)= \infty$.) 

We claim that either $\max_{|x|=r} f(x)> h(r)$ or $\max_{|x|=r} f(x) \le h(r)$ for all $r$ sufficiently large. Suppose the first statement is false. Then there exists a sequence of positive numbers $\{ r_k\} $ with $r_k \to \infty$ as $k\to \infty$ so that $\max_{|x|=r_k} f(x)\le h(r_k)$. Then by Theorem~\ref{th:global-maximum}, we have $\max_{|x|=r} f(x) \le h(r)$ for all $r$ sufficiently large. This proves the claim.

Suppose $\max_{|x|=r} f(x) > h(r)$ for all $r$ sufficiently large. By the assumption $\min_{|x| = r} f + C\ge \max_{|x|=r} f(x)$ for $r$ sufficiently large, we have for all $r \ge (2m)^{\frac{1}{n-2}}$
\[
	\min_{|x|=r} f(x) > h(r)-C',
\] 
for some constant $C' > 0$. Hence,  $f(x) > h(x) - C'$ for all $x \in \mathbb{R}^n\setminus  B_{(2m)^{\frac{1}{n-2}}}$. We then continuously decrease $C'$ until $f(x) = h(x) - C'$ for the first time. Then either Theorem~\ref{th:strong} or Theorem~\ref{th:boundary} implies $f(x) \equiv h(x)$, which leads a contradiction. Hence, $\max_{|x|=r} f(x) \le h(r)$ for all $r$ sufficiently large. Let $v_{\lambda}=h+\lambda$ for $\lambda>0$ sufficiently large. Then we continuously decrease $\lambda$. The graph of $v_{\lambda}$ approaches the graph of $f$ from above, until they touch for the first time.  By either Theorem~\ref{th:strong} or Theorem~\ref{th:boundary}, we conclude that $f \equiv h$.
\qed 
\end{pfthm}

\begin{pfSchwarzschild}
Suppose to the contrary that there is a complete $C^{n+1}$ hypersurface $M$ of one end with zero scalar curvature in $\mathbb{R}^{n+1}$ which is identical to the Schwarzschild solution $h$ (given in Proposition~\ref{pr:Schwarzschild}) of $m>0$ outside a compact set. We consider the graph of $h-\lambda$ for some constant $\lambda>0$. For $\lambda\gg 1$, the graph of $h-\lambda$ has no intersection with $M$. Then we decrease $\lambda$ until the graph of $h-\lambda$ approaches to $M$ from below and touches $M$ at a point $p$ for the first time. 

Note that by \cite[Theorem 4]{Huang-Wu} the mean curvature of $M$ has a sign. Then by Proposition~\ref{pr:HHR} and the fact that the level set of $M$ passing through $p$ is mean convex near $p$ with respect to the inward unit normal, the mean curvature of $M$ near $p$ (with respect to the unit normal vector pointing away from the graph of $h$)  is non-negative. Depending on that $p$ is either an interior point or a boundary point of the graph of $h-\lambda$, we apply either Theorem~\ref{th:strong} or Theorem~\ref{th:boundary} to conclude that $M$ is identical to the graph of $h$ outside $B_{(2m)^{\frac{1}{n-2}}}$ over $\mathbb{R}^n$. Then notice that $M$ must graphical in a neighborhood of $S_{(2m)^{\frac{1}{n-2}}}$ in $B_{(2m)^{\frac{1}{n-2}}}$, for otherwise $p$ cannot be the first touch point. Applying Theorem~\ref{th:strong-max} yields a contradiction. 
\qed
\end{pfSchwarzschild}

\end{document}